\documentclass[a4paper,12pt,reqno]{amsart}

\usepackage[utf8]{inputenc}
\usepackage{latexsym}
\usepackage{amssymb}
\usepackage{amsmath}
\usepackage{enumerate}
\usepackage{xcolor}
\usepackage{url}
\usepackage{mathrsfs}

\usepackage[normalem]{ulem}

\newtheorem{theorem}{Theorem}
\newtheorem{lemma}[theorem]{Lemma}

\newtheorem{corollary}[theorem]{Corollary}

\newcommand{\la}{\lambda}
\begin{document}

\title[Multiplicatively spectrum-preserving maps]{
Multiplicatively spectrum-preserving maps on $C^{*}$-algebras 
}
\author{
Michiya Mori
}
\address{
Graduate School of Mathematical Sciences, The University of Tokyo, 3-8-1 Komaba, Meguro-ku, Tokyo, 153-8914, Japan; Interdisciplinary Theoretical and Mathematical Sciences Program (iTHEMS), RIKEN, 2-1 Hirosawa, Wako, Saitama, 351-0198, Japan.
}
\email{mmori@ms.u-tokyo.ac.jp
}

\author{
Shiho Oi
}
\address{
Department of Mathematics, Faculty of Science, 
Niigata University, Niigata 950-2181, Japan.
}
\email{shiho-oi@math.sc.niigata-u.ac.jp
}


\maketitle
%\maketitle\textbf{}

\begin{abstract}
We study surjective maps between the sets of all self-adjoint elements of unital $C^*$-algebras which satisfy the multiplicatively spectrum-preserving property. We show that such maps are characterized by Jordan isomorphisms and central symmetries. This is an answer to a problem posed by Moln\'ar.
\end{abstract}

%%%%%%%%%%%%%%%%%%%%%%%%%%%%%%%%%%%%%%%%%%%%%
%%%%%%%%%%%%%%%%%%%%%%%%%%%%%%%%%%%%%%%%%%%%%
%%%%%%%%%%%%%%%%%%%%%%%%%%%%%%%%%%%%%%%%%%%%%

%%%%%%%%%%%%%%%%%%%%%%%%%%%%%%%%%%%%%%%%%%%%%
%%%%%%%%%%%%%%%%%%%%%%%%%%%%%%%%%%%%%%%%%%%%%
%%%%%%%%%%%%%%%%%%%%%%%%%%%%%%%%%%%%%%%%%%%%%

%%%%%%%%%%%%%%%%%%%%%%
%%%%%%%%%%%%%%%%%%%%%%
%%%%%%%%%%%%%%%%%%%%%%

%%%%%%%%%%%%%%%%%%%%%%
\section{Introduction}
Let $A$ and $B$ be unital $C^{*}$-algebras and let $\sigma(\cdot)$ denote the spectrum of elements. 
For any $a,b \in A$, we define the Jordan product $\circ$ by $a \circ b=\frac{1}{2}(ab+ba)$. We denote the set of all self-adjoint elements of $A$ (resp.\ $B$) by $A_{sa}$ (resp.\ $B_{sa}$). As we have $a \circ b \in A_{sa}$ for any $a,b \in A_{sa}$, 
the self-adjoint part $A_{sa}$ of any $C^{*}$-algebra $A$ is a Jordan algebra, which is a real commutative (non-associative) algebra whose product $\circ$ satisfies $a \circ (b\circ a^2)=(a \circ b)\circ a^2$.  We call a real-linear map $J\colon A_{sa} \to B_{sa}$ a Jordan homomorphism if $J(a\circ b)=J(a) \circ J(b)$ holds for any $a,b \in A_{sa}$. Moreover, if $J\colon A_{sa} \to B_{sa}$ is bijective, then it is called a Jordan isomorphism. For any Jordan homomorphism $J\colon A_{sa} \to B_{sa}$, one may uniquely obtain the complexification of $J$, which is a complex-linear map $J_{\mathbb{C}}$ from $A$ into $B$ that satisfies $J=J_{\mathbb{C}}$ on $A_{sa}$.
Then $J_{\mathbb{C}}$ is a Jordan $*$-isomorphism, that is, it satisfies $J_{\mathbb{C}}(a \circ b)=J_{\mathbb{C}}(a) \circ J_{\mathbb{C}}(b)$, and $J_{\mathbb{C}}(a^{*})=J_{\mathbb{C}}(a)^{*}$ for any $a,b \in A$. 
Jordan $*$-isomorphisms appear ubiquitously when studying mappings between $C^{*}$-algebras preserving a certain structure. See for example \cite{molnar2}.
On the other hand, the quantum-mechanical observables are expected to be self-adjoint in quantum theory. From this perspective, studying Jordan isomorphisms between the self-adjoint parts of $C^{*}$-algebras is equally interesting. 

St\o rmer \cite[Theorem 3.3]{stormer} showed that any Jordan $*$-homomorphism between two unital $C^{*}$-algebras is the sum of a $*$-homomorphism and a $*$-anti-homomorphism. This implies that any Jordan isomorphism $J\colon A_{sa} \to B_{sa}$ satisfies
\begin{equation}\label{jajb}
\sigma(J(a)J(b))=\sigma(ab), \quad a,b \in A_{sa}.
\end{equation}
For a detailed proof, see the first part of the proof of Theorem 2.2 in \cite{molnar}.
We call \eqref{jajb} the multiplicatively spectrum-preserving property of a Jordan isomorphism. 

Moln\'ar initiated the study of the multiplicatively spectrum-preserving property in \cite{molnar0}. In \cite{molnar},  he considered whether a surjection (not assumed to be linear)  between self-adjoint parts of unital $C^{*}$-algebras which has the multiplicatively spectrum-preserving property induces a Jordan isomorphism or not. 
He obtained a positive solution in the contexts of von Neumann algebras (\cite[Theorem 2.2]{molnar}) and standard $C^{*}$-algebras (\cite[Theorem 2.5]{molnar}). %This is a complete characterization of multiplicatively spectrum-preserving maps. 
Indeed, in order to prove the theorem \cite[Theorem 2.2]{molnar}, he applied the lemmas \cite[Lemma 2.3]{molnar} (\cite[Lemma 22]{molnar2}) and \cite[Lemma 2.4]{molnar} (\cite[Lemma 23]{molnar2}), which provide spectral characterizations of elements. However, he noted that the proofs of the lemmas were valid for only von Neumann algebras. Consequently, he posed a problem of whether the theorem \cite[Theorem 2.2]{molnar} can survive for general $C^*$-algebras (see \cite[p.194]{molnar}, \cite[Problem 4]{molnar2}). 
It can be seen from \cite[p.35]{molnar2} that the answer is yes if
one may obtain extensions of the lemmas \cite[Lemma 2.3]{molnar} (\cite[Lemma 22]{molnar2}) and \cite[Lemma 2.4]{molnar} (\cite[Lemma 23]{molnar2}) to general $C^{*}$-algebras. The purpose of this paper is to actually carry it out in Lemma \ref{lemma2} and Lemma \ref{spectrum} below. This gives a positive answer to the problem, which reads as follows.
Recall that a self-adjoint unitary is called a symmetry.

\begin{theorem}\label{main}
Let $A$ and $B$ be unital $C^{*}$-algebras. Suppose that a surjective map $\phi\colon A_{sa} \to B_{sa}$ satisfies 
\[
\sigma(\phi(a)\phi(b))=\sigma(ab),\ \ a,b \in A_{sa}.
\]  Then there is a central symmetry $s \in B_{sa}$ and a Jordan isomorphism $J\colon A_{sa} \to B_{sa}$ such that 
\[
\phi(a)=sJ(a), \quad a \in A_{sa}.
\]
\end{theorem}
Note that the problem we study is one of two problems posed by Moln\'ar in \cite[p.194]{molnar}. Another problem has an affirmative solution, too (see \cite[p.36]{molnar2} and \cite[Theorem 5.5]{HatoriOi}). 
Also note that multiplicatively spectrum-preserving mappings defined on the whole algebra (instead of the self-adjoint part) are studied in \cite{BMS} in the setting of unital semisimple Banach algebras.
%\begin{remark}
%Note that the second problem posed by Moln\'ar in \cite[p.194]{molnar} was solved affirmatively in \cite[Theorem 5.5]{HatoriOi}. 
%\end{remark}

\section{Results}

We present notations to be used throughout the paper. For a unital $C^{*}$-algebra $A$, $1$ stands for its unit. 
The symbols $A_{+}$ and $A_{+}^{-1}$ denote the set of all positive elements and the set of all positive invertible elements of $A$, respectively. 
For $x \in A_{sa}$, let $x=x_{+}-x_{-}$ denote the decomposition into its positive and negative parts. Note that $x_{+}x_{-}=x_{-}x_{+}=0$ holds. 
%For $x \in A$, the spectrum radius of $x$ is written as $\rho(x)$.
If $x \in A_+$, then $x^{1/2}\in A_+$ stands for its positive square root.

%%%%%%%%%%%%%%%%%%%%%%%%%%%%%%%%%%%%%%%%%%%%%
%%%%%%%%%%%%%%%%%%%%%%%%%%%%%%%%%%%%%%%%%%%%%
%%%%%%%%%%%%%%%%%%%%%%%%%%%%%%%%%%%

\begin{lemma}\label{lemma0}
Let $x \in A_+$ and $y\in A$. 
If $xy=0$, then $x^{1/2}y=0$. 
If $yx=0$, then $yx^{1/2}=0$. 
\end{lemma}
\begin{proof}
This follows from the basic fact that there is a sequence of polynomials $P_n$ satisfying $P_n(0)=0$ and $P_n(x)\to x^{1/2}$ as $n\to \infty$. 
\end{proof}

Throughout the paper, we use the following well-known fact several times. 
\begin{equation}\label{abba}
\sigma(ab)\setminus\{0\}=\sigma(ba)\setminus\{0\} \text{ whenever }a,b\in A. 
\end{equation}
This is called Jacobson's lemma.

The next lemma is an extension of \cite[Lemma 2.3]{molnar} (\cite[Lemma 22]{molnar2}). This might be of independent interest. 

\begin{lemma}\label{lemma2}
Let $x \in A_{sa}$. Then the following two conditions are equivalent. 
\begin{itemize}
\item $\sigma(xy) \subset \mathbb{R}$ for every $y \in A_{sa}$. 
\item $x_{+}yx_{-}=0$ for every $y \in A$.
\end{itemize}
\end{lemma}
\begin{proof}
Firstly we assume that $x_{+}yx_{-}=0$ for every $y \in A$. 
Let $y \in A_{sa}$.
By Lemma \ref{lemma0}, we have $x_{+}-x_{-}=(x_{+}^{1/2}+x_{-}^{1/2})(x_{+}^{1/2}-x_{-}^{1/2})$ and $x_{+}^{1/2}yx_{-}^{1/2}=0$. 
We also have $x_{-}^{1/2}yx_{+}^{1/2}=(x_{+}^{1/2}yx_{-}^{1/2})^*=0$.
Thus we get 
\[
xy=(x_{+}-x_{-})y=(x_{+}^{1/2}+x_{-}^{1/2})(x_{+}^{1/2}-x_{-}^{1/2})y
\]
and 
\[
(x_{+}^{1/2}-x_{-}^{1/2})y(x_{+}^{1/2}+x_{-}^{1/2})=x_{+}^{1/2}yx_{+}^{1/2}-x_{-}^{1/2}yx_{-}^{1/2}\in A_{sa}.
\]
Combining these with \eqref{abba}, we get $\sigma(xy)\subset\mathbb{R}$. 

Conversely, assume that $\sigma(xy) \subset \mathbb{R}$ for every $y \in A_{sa}$.
Let $y \in A$. We prove $x_{+}yx_{-}=0$.
It suffices to consider the case $y \in A_{sa}$. 
In this case, we have $z:=x_{+}^{1/2}yx_{-}^{1/2}+x_{-}^{1/2}yx_{+}^{1/2}\in A_{sa}$. 
From $\sigma(xz)\subset \mathbb{R}$,
\[
xz=(x_{+}-x_{-})z=(x_{+}^{1/2}+x_{-}^{1/2})(x_{+}^{1/2}-x_{-}^{1/2})z,
\]
\[
(x_{+}^{1/2}-x_{-}^{1/2})z(x_{+}^{1/2}+x_{-}^{1/2})=x_{+}yx_{-}-x_{-}yx_{+}\in iA_{sa},
\]
and \eqref{abba}, we get $\sigma(x_{+}yx_{-}-x_{-}yx_{+})=\{0\}$ and $x_{+}yx_{-}=x_{-}yx_{+}$. 
Hence $x_{+}yx_{-}$ is self-adjoint and 
\[
(x_{+}yx_{-})^2=(x_{+}yx_{-})(x_{+}yx_{-})=x_{+}y\cdot 0\cdot yx_{-}=0,
\]
so we obtain $x_{+}yx_{-}=0$ as desired.
\end{proof}

Our proofs of Lemma \ref{lemma1} and Lemma \ref{spectrum} are similar to the proof of Theorem 2.2 in \cite{BB}, which shows that for any semisimple Banach algebra $X$ if $\sigma(ax)=\sigma(bx)$ for all elements $x \in X$ then $a=b$ holds.
\begin{lemma}\label{lemma1} Let $a, b \in A_{sa}$. Suppose that  $\sigma(ay)=\sigma(by)$ for any $y \in A_{+}^{-1}$. Let $x \in A_{sa}$ and let $\la$ be a real number with $|\la| > \lVert x\rVert$. Then $\la \in \sigma(x+a)$  if and only if $\la \in \sigma(x+b)$. 
\end{lemma}
\begin{proof}
Suppose that $\la>\lVert x\rVert$. Then  
$\la-x \in A_{+}^{-1}$. We have $\la-(x+a)=(\la-x)[1-(\la-x)^{-1}a]$. This implies that $\la \in \sigma(x+a) \iff 1 \in \sigma((\la-x)^{-1}a) \iff 1 \in \sigma((\la-x)^{-1}b) \iff \la \in \sigma(x+b)$. 
Suppose that $-\la > \lVert x\rVert$. Then 
$-\la+x \in A_{+}^{-1}$. Thus the conclusion follows from a similar procedure as described above.
\end{proof}

\begin{corollary}\label{lemma5} Let $a, b \in A_{sa}$. Suppose that  $\sigma(ay)=\sigma(by)$ for any $y \in A_{+}^{-1}$. Let $x \in A_{sa}$. Then either $\max\{\lVert x+a\rVert, \lVert x+b\rVert\}\leq \lVert x\rVert$ or $\lVert x+a\rVert=\lVert x+b\rVert$ holds.
\end{corollary}

The next lemma extends \cite[Lemma 2.4]{molnar} (\cite[Lemma 23]{molnar2}).

\begin{lemma}\label{spectrum}
Let $a, b \in A_{sa}$. If $\sigma(ay)=\sigma(by)$ for any $y \in A_{+}^{-1}$, then we have $a=b$.
\end{lemma}
\begin{proof}
We prove by induction that 
\begin{equation}\label{induction}
\lVert nb-(n-1)a\rVert \leq \lVert b\rVert \text{ and }\lVert n(b-a)\rVert\le \lVert b\rVert 
\end{equation}
for $n \ge 1$.
Applying Corollary \ref{lemma5} with $x=-b \in A_{sa}$, we get either $\lVert -b+a\rVert\leq \lVert -b\rVert$ or $\lVert -b+a\rVert=\lVert -b+b\rVert=0$. In both cases, we get  $\lVert b-a\rVert \leq \lVert b\rVert$.  
Thus we get \eqref{induction} when $n=1$. 

Let $k\geq1$ and suppose that \eqref{induction} holds with $n=k$. Applying Corollary \ref{lemma5} with $x=k(b-a) \in A_{sa}$, we get either
$\lVert k(b-a)+b\rVert\leq \lVert k(b-a)\rVert$ or $\lVert k(b-a)+a\rVert=\lVert k(b-a)+b\rVert$.
Therefore, in both cases, the assumption implies
\begin{equation}\label{induction1}
\lVert (k+1)b-ka\rVert \le \lVert b\rVert.
\end{equation} 
Applying Corollary \ref{lemma5} with  $x=ka-(k+1)b \in A_{sa}$, we have either
$\lVert ka-(k+1)b+a\rVert\leq \lVert ka-(k+1)b\rVert$ or $\lVert ka-(k+1)b+a\rVert=\lVert ka-(k+1)b+b\rVert$.
In both cases, the assumption together with \eqref{induction1} shows 
\[
\lVert (k+1)(b-a)\rVert  \le \lVert b\rVert.
\]

Thus we have obtained the equation $\lVert n(b-a)\rVert \le \lVert b\rVert$ for any $n \geq 1$. This implies  $a=b$.
\end{proof}

Having obtained Lemma \ref{lemma2} and Lemma \ref{spectrum}, we are ready to prove the main theorem of this paper according to the observations of \cite[p.35]{molnar2}.    
Although the proof is very similar to that of \cite[Theorem 2.2]{molnar}, we present it for the sake of readability.
\begin{proof}[Proof of Theorem \ref{main}]
Let $s=\phi(1)$. Then $s \in B_{sa}$. We have $\sigma(s^2)%=\sigama(\phi(1)\phi(1))
=\sigma(1)=\{1\}$, thus we obtain that $\sigma(s) \subset \{1,-1\}$, so $s$ is a symmetry. As $\phi$ is a surjection, for any $y \in B_{sa}$ there is  $x \in A_{sa}$ such that $\phi(x)=y$. We have $\sigma(sy)=\sigma(\phi(1)\phi(x))=\sigma(x) \subset \mathbb{R}$. By Lemma \ref{lemma2}, we obtain $s_{+}ys_{-}=0$ for any $y \in B$. For any $y \in B_{sa}$, we get $s_{+}ys_{-}=0$ and $s_{-}ys_{+} = (s_{+}ys_{-})^*=0$, thus 
\[
sy=(s_{+}-s_{-})y(s_{+}+s_{-})=s_{+}ys_{+}-s_{-}ys_{-} \in B_{sa}.
\] 
It follows that $sy=(sy)^{*}=ys$. Hence $s$ is central. 

We define a map $T\colon A_{sa} \to B_{sa}$ by $T(x)=s\phi(x)$ for any $x \in A_{sa}$. 
For any $y \in B_{sa}$, as $sy \in B_{sa}$ and $\phi$ is surjective, there is $x \in A_{sa}$ such that $\phi(x)=sy$. Thus $T(x)=ssy=y$. 
Therefore, $T\colon A_{sa} \to B_{sa}$ is a surjection. We have  
\begin{equation}\label{multiplicative}
\sigma(T(x)T(y))=\sigma(s\phi(x)s\phi(y))=\sigma(\phi(x)\phi(y))=\sigma(xy)
\end{equation}
for any $x,y \in A_{sa}$.  
By taking $y=1$, we get $\sigma(T(x))=\sigma(x)$ for any $x \in A_{sa}$. 
It follows that $x\in A_{+}^{-1} \iff T(x)\in B_{+}^{-1}$.
Therefore, $T|_{A_{+}^{-1}}\colon A_{+}^{-1} \to B_{+}^{-1}$ is a surjective map which satisfies (\ref{multiplicative}). 

For any $x \in  A_{+}^{-1}$ we have 
\[
\{1\}=\sigma(xx^{-1})\stackrel{\eqref{multiplicative}}{=}\sigma(T(x)T(x^{-1}))\stackrel{\eqref{abba}}{=}\sigma(T(x)^{1/2}T(x^{-1})T(x)^{1/2}).
\]
Thus we obtain $T(x)^{1/2}T(x^{-1})T(x)^{1/2}=1$ and \begin{equation}\label{-1}
T(x)^{-1}=T(x^{-1}).
\end{equation} 
It follows that
\[
\begin{split}
\sigma(x^{-{1/2}}yx^{-{1/2}}) \stackrel{\eqref{abba}}{=} \sigma(yx^{-1}) &\stackrel{\eqref{multiplicative}}{=} \sigma(T(y)T(x^{-1}))\\
&\stackrel{\eqref{-1}}{=} \sigma(T(y)T(x)^{-1})\stackrel{\eqref{abba}}{=}\sigma(T(x)^{-{1/2}}T(y)T(x)^{-{1/2}})
\end{split}
\]
for any pair $x,y \in A_{+}^{-1}$.
Combining this with a fact that the Thompson metric $d_{T}$ on $A_{+}^{-1}$ satisfies $d_{T}(x,y)=\|\log(x^{-{1/2}}yx^{-{1/2}})\|$ for any $x,y \in A_{+}^{-1}$, we have $T|_{A_{+}^{-1}}\colon A_{+}^{-1} \to B_{+}^{-1}$ is a surjective Thompson isometry with $T(1)=1$. For any $t>0$ and $x,y \in A_{+}^{-1}$, we see that $\sigma(T(x)T(y))=\sigma(xy)$ and $\sigma(txy)=\sigma(T(tx)T(y))$ by \eqref{multiplicative}, hence
\[
\sigma(tT(x)T(y))=\sigma(txy)=\sigma(T(tx)T(y)).
\]
Thus, Lemma \ref{spectrum} shows that $tT(x)=T(tx)$ for any $x \in  A_{+}^{-1}$. In other words, $T|_{A_{+}^{-1}}$ is positive homogeneous. 
By applying \cite[Theorem 9]{HatoriMolnar} (\cite[Theorem 8]{molnar2}) with $T|_{A_{+}^{-1}}$, we see that there is a Jordan $*$-isomorphism $J\colon A \to B$ such that $T=J$ on $A_{+}^{-1}$. Fix $a \in A_{sa}$. For any $y \in B_{+}^{-1}$, there is $x \in A_{+}^{-1}$ such that $T(x)=J(x)=y$. We have 
\[
\sigma(J(a)y)=\sigma(J(a)J(x))\stackrel{\eqref{jajb}}{=}\sigma(ax)\stackrel{\eqref{multiplicative}}{=}\sigma(T(a)T(x))=\sigma(T(a)y).
\]
Since $J(a), T(a) \in B_{sa}$,  Lemma \ref{spectrum} implies that  $J(a)=T(a)$. Thus $\phi(a)=sJ(a)$ for any $a \in A_{sa}$. 
\end{proof}
%%%%%%%%%%%%%%%%%%%%%%%%%%%%%%%%%%%%%%%%%%%%

%%%%%%%%%%%%%%%%%%%%%%%%%%%%%%%%%%%%%%%%%%%%%%%%%%%%%%%%%%%%%%%%%%%%%%%%%%%%%%%%%%%%%%%%%%%%%%%%%%%%%%%%%%%%%%%%%%%%%%%%%%%%%%%%%%%%%%%%%%%%%%%%%%%%%%%%%%%%%%%%%%%%%%%%%%%%%%%%%%%%%%%%%%%%%%%%%%%%%%%
\subsection*{Acknowledgments}
The first and the second authors were supported by JSPS KAKENHI Grant Numbers 22K13934 and 21K13804, respectively. 
The authors appreciate Professors Masaki Izumi and Narutaka Ozawa for giving them an opportunity to visit Kyoto, where this joint work began.

%%%%%%%%%%%%%%%%%%%%%%%%%%%%%%%%%%%%%%%%%%%%%%%%%%%%%%%%%%%%%%%%%%%%%%%%%%%%%%%%%%%%%%%%%%%%%%%%%%%%%%%%%%%%%%%%%%%%%%%%%%%%%%%%%%%%%%%%%%%%%%%%%%%%%%%%%%%%%%%%%%%%%%%%%%%%%%%%%%%%%%%%%%%%%%


\begin{thebibliography}{99}
\bibitem{BMS}
A.~Bourhim, J.~Mashreghi, and A.~Stepanyan,
\emph{
Maps between Banach algebras preserving the spectrum}, 
Arch. Math. (Basel), \textbf{107} (2016), no.6, 609--621.


\bibitem{BB}
G.~Braatvedt and R.~Brits,
\emph{
Uniqueness and spectral variation in Banach algebras},
Quaest. Math., \textbf{36} (2013), No. 2, 155--165.

\bibitem{HatoriMolnar}
O.~Hatori and L.~Moln\'ar, 
\emph{
Isometries of the unitary groups and Thompson isometries of the spaces of invertible positive elements in $C^{*}$-algebras}, 
J. Math. Anal. Appl., \textbf{409} (2014), no.1, 158--167.


\bibitem{HatoriOi}
O.~Hatori and S.~Oi,
\emph{
Non-linear characterization of Jordan $*$-isomorphisms via maps on positive cones of $C^{*}$-algebras},
preprint, arXiv:2403.07341.


\bibitem{molnar0}
L.~Moln\'ar,
\emph{
Some characterizations of the automorphisms of $B(H)$ and $C(X)$}, Proc. Amer. Math. Soc., \textbf{130} (2002), no.1, 111--120. 

\bibitem{molnar}
L.~Moln\'ar,
\emph{
Spectral characterization of Jordan-Segal isomorphisms of quantum observables}, J. Oper. Theory, \textbf{83} (2020), 139--177.

\bibitem{molnar2} 
L.~Moln\'ar, 
\emph{Jordan isomorphisms as preservers}, Linear and multilinear algebra and function spaces, 19--41,
Contemp. Math., 750, Centre Rech. Math. Proc., Amer. Math. Soc., [Providence], RI (2020).


\bibitem{stormer}
E.~St\o rmer,
\emph{
On the Jordan structure of $C^{*}$-algebras}, 
Trans. Am. Math. Soc., \textbf{120} (1965), 438--447.
\end{thebibliography}
\end{document}